\documentclass[preprint, 12pt]{elsarticle}

\usepackage{amssymb}

\usepackage[top=1in, bottom=1in, left=1in, right=1in]{geometry}

\usepackage[dvipdfm,bookmarks=true,bookmarksnumbered=true,bookmarkstype=toc]{hyperref}   
\usepackage{amsmath, amsthm, amsfonts}
\usepackage{amscd}
\usepackage{verbatim}

\numberwithin{equation}{section}

\numberwithin{figure}{section}

\newtheorem{theorem}{Theorem}[section]
\newtheorem{corollary}[theorem]{Corollary}
\newtheorem{lemma}[theorem]{Lemma}

\newtheorem{definition}[theorem]{Definition}

\newtheorem{remark}[theorem]{Remark}

\newtheorem{remarks}[theorem]{Remarks}
\newtheorem{proposition}[theorem]{Proposition}

\newtheorem{conjecture}[theorem]{Conjecture}

\newtheorem{example}[theorem]{Example}

\newcommand{\be}{\begin{equation}}

\newcommand{\core}{C_0^{\infty}(\Omega)}
\newcommand{\laplace}{\Delta}
\newcommand{\pl}{\laplace_p}
\newcommand{\grad}{\nabla}
\newcommand{\pd}{\partial}
\newcommand{\bo}{\pd}

\newcommand{\sm}{\setminus}

\newcommand{\diver}{\mathrm{div}\,}

\newcommand{\resc}{\mathcal{D}}
\newcommand{\fundot}{\circ}

\newcommand{\beq}{\begin{equation}}
\newcommand{\eeq}{\end{equation}}
\newcommand{\beqa}{\begin{eqnarray}}
\newcommand{\eeqa}{\end{eqnarray}}
\newcommand{\beqanl}{\begin{eqnarray*}}
\newcommand{\eeqanl}{\end{eqnarray*}}
\newcommand{\bs}{\begin{sub}}
\newcommand{\es}{\end{sub}}
\newcommand{\bsn}{\begin{subn}}
\newcommand{\esn}{\end{subn}}
\newcommand{\bea}{\begin{eqnarray}}
\newcommand{\eea}{\end{eqnarray}}
\newcommand{\bean}{\begin{eqnarray*}}
\newcommand{\eean}{\end{eqnarray*}}
\newcommand{\BA}[1]{\begin{array}{#1}}
\newcommand{\EA}{\end{array}}

\newlength{\wex}  \newlength{\hex}
\newcommand{\dt}{\,\mathrm{d}t}

\newcommand{\ds}{\,\mathrm{d}s}

\def\ga{\alpha}     \def\gb{\beta}       
             
                         \def\vge{\varepsilon}
           
            \def\gl{\lambda}
\def\gm{\mu}                 
    \def\gr{\rho}        
       \def\gt{\tau}
      
                \def\gz{\zeta}

\def\Gw{\Omega}              

\def\squarebox#1{\hbox to #1{\hfill\vbox to #1{\vfill}}}

\newcommand{\lsim}{\precsim}
\newcommand{\gsim}{\succsim}

\begin{document}

\begin{frontmatter}

\title{Isolated singularities of positive solutions of $p$-Laplacian type equations in $\mathbb{R}^d$}

\author{Martin Fraas}
\address{Department of Physics, Technion - Israel Institute of
Technology,  Haifa, Israel} \ead{martin.fraas@gmail.com}

\author{Yehuda Pinchover}
\address{Department of Mathematics, Technion - Israel Institute of
Technology,  Haifa, Israel\\[1cm]
Dedicated to the memory of Vitali Liskevich} \ead{pincho@techunix.technion.ac.il}

\begin{abstract}
We study the behavior of positive solutions of $p$-Laplacian type elliptic equations of the form
 \begin{equation*}
Q'(u):=-\Delta_p (u) + V |u|^{p-2} u = 0 \qquad \mbox{ in } \Omega\setminus \{\zeta\}
\end{equation*}
near an isolated singular point $\zeta \in \Omega\cup\{\infty\}$, where $1<p<\infty$, and $\Omega$ is a domain in $\mathbb{R}^d$ with $d>1$.
We obtain removable singularity theorems for positive solutions near $\zeta$. In particular, using a new three-spheres theorems for certain solutions of the above equation near $\zeta$ we prove that if $V$ belongs to a certain Kato class near $\zeta$ and $p>d$ (respectively, $p<d$), then any positive
solution $u$ of the equation $Q'(u)=0$ in a punctured neighborhood of $\zeta=0$ (respectively, $\zeta=\infty$) is in fact continuous at $\zeta$.  Under further assumptions we find the asymptotic behavior of $u$ near $\zeta$.
\\[4mm]
\end{abstract}
\end{frontmatter}

\section{Introduction}\label{sec1}

The aim of this paper is to study the behavior of positive solutions of the quasilinear elliptic equation
\begin{equation} \label{eq:1}
Q'_V(u)=Q'(u):=-\pl (u) + V |u|^{p-2} u = 0 \qquad \mbox{ in } \Omega\sm \{\zeta\}
\end{equation}
near an {\em isolated} singular point $\zeta \in \Omega\cup\{\infty\}$, where
$\pl(u):=\!\diver(|\nabla u|^{p-2}\nabla u)$ is the $p$-Laplacian with $1<p<\infty$, $\Omega$ is a domain in $\mathbb{R}^d$ with
 $d>1$, and  $V$ is a real function satisfying
\begin{equation}\label{eq:V}
V\in  L^\infty_{\mathrm{loc}}(\Omega\sm \{\gz\}).
\end{equation}
We are interested in the local behavior of solutions near $\zeta$, and in removable singularity theorems for such equations. Therefore, unless otherwise stated, we assume
that either $\zeta = 0$ and $\Omega=B_R$ is the open ball of radius $R>0$ centered at the origin, or
 $\zeta = \infty$ and $\Gw$ is the exterior domain $\Omega=B_R^*:=\mathbb{R}^d\sm \overline{B_R}$.

 The different behaviors of positive solutions near isolated singularities in different cases are well illustrated when $V=0$ and $Q'$ is the $p$-Laplacian.
\begin{example}\label{ex:pl}{\em
Let $\Omega=\mathbb{R}^d\sm \{0\}$, $ d > 1$, and denote
\begin{equation}\label{ga*}
\alpha^* := \frac{p-d}{p-1}\,.
\end{equation}
 Define the ``fundamental solution" of the $p\,$-Laplacian by
\begin{equation} \label{eq:free}
v_{\alpha^*}(x): = \left\{
\begin{array}{lr} |x|^{\alpha^*} & p \neq d, \\[2mm]
		  |\log(|x|)| & p = d ,\;\;| x|\neq 1.
\end{array} \right.
\end{equation}
Then for any $A>0$ and $B\geq 0$ the positive function $u(x):=A v_{\alpha^*}(x) + B$ is $p$-harmonic in punctured neighborhoods of $\zeta=0$ and $\zeta=\infty$.
Indeed, let $v$ be a radial function, and denote $r:=|x|$, $v':=\pd v/\pd r$, then the radial $p\,$-Laplacian is given by
\begin{equation}\label{eq:green}
-\pl (v) =-\frac{1}{r^{d-1}}\left( r^{d-1}|v'|^{p-2} v'  \right)' =
-|v'|^{p-2}\left[(p-1)v''+\frac{d-1}{r}v'\right].
\end{equation}
Consequently, for $v_\alpha(r):=r^\ga$ with $\ga\neq 0$ we have
\begin{equation}\label{eq:green1}
-\pl (v_\alpha) =
-|\alpha|^{p-2} \alpha
\left[\alpha (p-1) + d -p \right]r^{\alpha(p-1) - p}.
\end{equation}
Substituting $\ga=\ga^*$, we obtain for $p\neq d$ that
$$-\pl (u(x))=-\pl (A v_{\alpha^*}(x) + B)=0 \qquad \forall x\in \mathbb{R}^d\sm \{0\}.$$
The case $p=d$ is similar (for $| x|\neq 0,1$).

Consequently,
\begin{equation}\label{eq:limplp}
\lim_{x \to 0} u(x) = \left\{
\begin{array}{lr} \infty & p \leq d, \\[2mm]
		  B & p > d,
\end{array} \right. \qquad
\lim_{x \to \infty} u(x) = \left\{
\begin{array}{lr} B & p < d, \\[2mm]
		  \infty & p \geq d.
\end{array} \right.
\end{equation}
 }
\end{example}
\begin{definition}\label{def:class}{\em
We call the cases $\zeta = 0$ with $p \leq d$, and  $\zeta = \infty$ with $p \geq d$
the {\em classical cases}. The complementary cases (where $\gz$ has a nonzero capacity with respect to the $p\,$-Laplacian \cite[Example~2.22]{Martio}) are called the {\em nonclassical cases}. Note that in the classical cases we have $\lim_{x \to \gz}v_{\alpha^*}(x)=\infty$.
   }
\end{definition}

Our aim is to find an appropriate class of potentials for which positive solutions of the equation $Q'(u)=0$ near $\gz$ exhibits similar behavior to that of the $p\,$-Laplacian. We consider the following classes of singular potentials (cf. \cite[definitions~1.1 and 2.1]{PF}).
\begin{definition}\label{def:Fuchs} {\em
Let $\Gw$ be a domain, and let $\zeta \in \Omega \cup \{\infty\}$, where $\zeta=0$ or $\zeta=\infty$. Let $V\in L^\infty_{\mathrm{loc}}(\Omega\sm \{\gz\})$.  We say that (\ref{eq:1}) has a {\em Fuchsian type singularity at $\zeta$} if there exists a positive constant $C$ such that $V$ satisfies
\begin{equation}\label{eq:FuchsCon}
 |x|^p |V(x)| \leq C \qquad \mbox{near $\gz$}.
    \end{equation}
}
\end{definition}

For $R>0$ let $V_R$ be the {\em scaled potential} defined by
\begin{equation} \label{eq:wfuchs}
 V_R(x) := R^p V(R x) \qquad  x \in \Omega/R.
 \end{equation}
Let $\{R_n\}\subset \mathbb{R}_+$ be a sequence satisfying   $R_n \to \zeta$
(where $\zeta$ is either $0$ or $\infty$)
such that \begin{equation}\label{eq:rescalepotential}
V_{R_n}\underset{n\to \infty}{\longrightarrow} W \quad \mbox{ in  the weak$^*$ topology of  $L_{\mathrm{loc}}^{\infty}(Y)$,}
\end{equation}
 where $Y:=\lim_{n\to\infty} \Omega/R_n = \mathbb{R}^d\setminus\{0\}$.
Define the {\em limiting dilated equation with respect to} (\ref{eq:1}) (and the sequence $\{R_n\}$) as
\begin{equation}
\resc^{\{R_n\}}(Q)(w) := -\pl(w) + W |w|^{p-2} w=0 \qquad \mbox{ on } Y.
\end{equation}

\begin{definition}\label{def:weak_Fuchs} {\em
Let $V\in L^\infty_{\mathrm{loc}}(\Omega \setminus\{\zeta\})$.  We say that $V$ has a {\em weak Fuchsian singularity at $\zeta$}  if inequality \eqref{eq:FuchsCon} is satisfied, and in addition, there exist  $m$ sequences $\{R_n^{(i)}\}_{n=1}^\infty \subset \mathbb{R}_+$, $1\leq i\leq m$,  satisfying $R_n^{(i)} \to \zeta^{(i)}$, where $\zeta^{(1)} = \zeta$, and $\zeta^{(i)} = 0$ or
$\zeta^{(i)} = \infty$
 for $2\leq i \leq m$,  such that
\begin{equation}\label{eq:weakfuchs}
\resc^{\{R_n^{(m)}\}} \fundot \dots  \fundot \resc^{\{R_n^{(1)}\}}(Q)(w) = -\pl(w) \qquad \mbox{ on } Y.
\end{equation}
 }
\end{definition}
\begin{remark}\label{rem_closure}
{\em Let $V  \in L^\infty_{\mathrm{loc}}(\mathbb{R}^d\setminus\{0\})$ and let $SC_V \subset L^\infty_{\mathrm{loc}}(\mathbb{R}^d\setminus\{0\})$ be the set of the scaled potentials $SC_V := \{ R^p V(R x),\,R >0\}$. If $V$ is a weak Fuchsian potential, then $W=0$ belongs to a weak* closure of $SC_V$. We do not know whether the reverse implication holds true.
}
\end{remark}
Throughout the paper we assume that near $\gz$ the function $V$ is of the form:
\begin{equation}\label{CondC}
   \qquad |V(x)| \leq  G(|x|) = \frac{g(|x|)}{|x|^p} \,,
\end{equation}
 where  $g : \mathbb{R}_+ \to \mathbb{R}_+$ is a positive bounded continuous function which satisfies some further conditions.

 For functions satisfying \eqref{CondC}, we say that $V$ belongs to {\em Kato class near} $\gz$  if $g$ satisfies the following  integral condition
 \begin{equation*}\tag{\textbf{C1}}
   \left| \int_\zeta^1 t^{\alpha^*}\left|\int_a^t \frac{g(s)}{s^p} \,s^{d-1} \ds \right|^{\frac{1}{p-1}} \frac{\dt}{t} \right| = \left| \int_\zeta^1 \left|t^{1-d}\int_a^t \frac{g(s)}{s^p} \,s^{d-1} \ds \right|^{\frac{1}{p-1}} \dt \right|< \infty,
 \end{equation*}
where $a$ is chosen to make the class of admissible singular potentials near $\gz$ almost optimal. So, $a=\gz$ in the classical cases and $a=1$ in the nonclassical cases, i.e. $a$ is given by
\vskip 4mm
\begin{center}
\begin{tabular}{|c|c|c|c|}
\hline
  & $p < d$ & $p>d$ & $p=d$ \\[3mm]
\hline
$\zeta = 0$ & $a = 0$ & $a = 1$ & $a=0$, \\[3mm]
\hline
$\zeta = \infty$ & $a = 1$ & $a = \infty$ & $ a =\infty $.\\[3mm]
\hline
\end{tabular}
\end{center}
\begin{remark}\label{rem_Kato}{\em 
The quasilinear Kato class in $\Omega$ is defined and studied in \cite{JV,LS,LSS,MH} using the Wolff potential. We note that if $V \in L^\infty(\Omega \setminus\{\zeta\})$ belongs to Kato class near $\zeta$, then $V$ belongs to the Kato class in $\Omega$.
 }
 \end{remark}
Some of our results are valid under the following closely related {\em Dini condition} at $\gz$
\begin{align*}\tag{\textbf{C2}}
  \begin{array}{ll}
   \displaystyle{\left| \int_\zeta^1 \frac{g(s)}{s} \,\ds \right|<\infty}  & p\neq d, \\[6mm]
   \displaystyle{\left| \int_\zeta^1 \frac{g(s)}{s} |\log s|^{\,d-1} \,\ds \right| < \infty} & p = d.
  \end{array}
\end{align*}
In Section~\ref{sec2} we prove the following lemma.
\begin{lemma}\label{C1C@weakFuchs}
If $V\in  L^\infty_{\mathrm{loc}}(\Omega\sm \{\gz\})$ satisfies either Condition~(\textbf{C1}) or Condition~(\textbf{C2}) near $\gz$, then $V$ has a weak Fuchsian singularity at $\zeta$. Moreover,
\begin{equation}\label{eq:rescalepotential1}
R^p V(R x)\underset{R\to \gz}{\longrightarrow} 0 \quad \mbox{ in  the weak$^*$ topology of  $L_{\mathrm{loc}}^{\infty}(\mathbb{R}^d\setminus\{0\})$.}
\end{equation}
\end{lemma}

\begin{remarks}\label{assumptions}{\em
1. Two prototypes of functions $g$ that satisfy both conditions (\textbf{C1}) and (\textbf{C2}) are $g(|x|)=|x|^\vge$ and $|\log(|x|)|^{-\gb}$
 if $\zeta=0$, or $g(|x|)=|x|^{-\vge}$ and $|\log(|x|)|^{-\gb}$ if $\zeta=\infty$, where  $\vge$ is any positive number and $\gb > (p-1)$ for $p \neq d$ or $\beta > p$ for $p=d$.

\vskip 3mm

2. If $V$ belongs to $L^{q}(\Omega \sm  \{\zeta\})$, then
without loss of generality we may assume that $G\in L^{q}(\Omega \sm  \{\zeta\})$, where $G(s) = g(s)|s|^{-p}$ satisfies \eqref{CondC}. Therefore, in the classical case  $p \leq d$ and $\zeta = 0$
(respectively, $p \geq d$ and $\zeta = \infty$), if $V \in L^{q}(\Omega \sm  \{\zeta\})$ with $q > d/p$ (respectively, $q <d/p$), then conditions (\textbf{C1}) and (\textbf{C2}) are satisfied near $\gz$ (cf. Theorem~\ref{thm:asympSerrin} and \cite{Ser}).

\vskip 3mm

3. In the nonclassical cases, the integrability of $V$ near $\gz$ clearly implies that $V$ satisfies conditions~(\textbf{C1}) and (\textbf{C2}) near $\gz$.
}
\end{remarks}

Under the above assumptions, positive solutions of the equation $Q'(u) = 0$ in $\Omega\sm \{\gz\}$ are
$C^{1,\ga}$-continuous and satisfy Harnack inequality in $\Omega\sm \{\gz\}$ \cite[Theorem 7.4.1]{Puc}.
Moreover, the equation $Q'(u)=0$ admits positive (weak) solutions near the singular point $\zeta$ (see Lemma~\ref{lem:pos} and Lemma~\ref{lem:green}).

The present paper is a continuation of our recent paper \cite{PF}. In that paper we proved {\em ratio-limit} theorems for {\em quotients} of two positive solutions of the equation $Q'_V(u) = 0$ near $\zeta$ if $V$ satisfies conditions~(\textbf{C1}) or (\textbf{C2}) near $\gz$. More precisely, under the  (weaker) assumption that $V$ has a weak Fuchsian singularity at the isolated singular point $\zeta$, it was shown that if $u$ and $v$  are two positive solutions of the equation $Q'(w) = 0$ in a punctured neighborhood of $\zeta$, then
 $$
\lim_{x \to \zeta} \frac{u(x)}{v(x)} \quad \mbox{exists;}
$$
the limit might be infinite. As a result, some positive Liouville theorems, and removable singularity theorems for the equation $Q'(w)=0$ in certain domains were obtained.

 The goal of the present paper is to study the precise behavior of a {\em single} positive solution $u$ defined near $\zeta$. One of the main results of the present work is that under assumption (\textbf{C1})
$$
\lim_{x \to \zeta} u(x) \quad \mbox{exists.}
$$
Moreover, in the nonclassical cases (where either $\gz=0$ and $p > d$, or $\gz=\infty$ and $p< d$) positive solutions of $Q'(u) = 0$ in $\Omega\sm \{\zeta\}$ are in fact continuous at the singular point $\zeta$ (see Theorem~\ref{thm:main}). Under further assumptions we obtain the exact asymptotic behavior   for the case $\lim_{x \to \zeta} u(x) = 0$ (see Theorem~\ref{thm:asymp3}).  The particular case $V = 0$ has been studied in \cite{KV,Serrin2}. The case $p = \infty$, $\zeta=0$,  and $Q'(u) = -\Delta_\infty(u)$ has been studied in \cite{bhat} and stimulated the present paper as well. We note that our result answers a question posed in \cite[Section 5]{PT} for the case $\zeta=0$ and $p>d$.

\begin{remark}\label{rem:Fuchs}{\em
Suppose that $V$ has a {\em Fuchsian type singularity} at $\gz$.
Then even in the linear case ($p=2$) positive solutions of $Q'(u) = 0$ in $\Omega\sm \{\zeta\}$ might not admit a limit at $\zeta$ (see \cite[Example~9.1]{P94}), albeit, a ratio limit theorem holds true near any Fuchsian singular point if $p=2$ \cite{P94}. Such a ratio limit theorem should also be true for $p\neq 2$ if \eqref{eq:FuchsCon} is satisfied (see \cite{PF} for partial results).
} \end{remark}

Isolated singularities have been studied  extensively in the past fifty years. Indeed, the case $p<d$ and $\zeta=\infty$ has been studied in \cite[and the references therein]{LSS,P}.
On the other hand, {\em interior} singularities of $p$-Laplacian type equations for the case $1<p \leq d$ have been studied by Serrin in \cite{Ser} who proved the following removable singularity result.

\begin{theorem}\label{thm:asympSerrin}
Let $1<p \leq d$, and assume that $V\in L^q_{\mathrm{loc}}(B_R)$ with $q > d/p $. Let $u$ be a positive solution of $Q'(u) =0$ in the punctured ball $B_R\sm \{0\}$.  Then either $u$ has a removable singularity at the origin, or
\begin{equation}\label{asymp:dgeqpSerrin}
u(x) \;\asymp\;
\left\{
  \begin{array}{ll}
   \quad |x|^{\alpha^*} & \hbox{ if }\; p<d,\\[2mm]
    -\log |x| & \hbox{ if }\; p=d.
  \end{array}
\right.
\end{equation}
\end{theorem}
We improve the above result and obtain the {\em asymptotic behavior} of positive solutions near $\gz=0$  (the proof
is given in the Section \ref{sec:asymp}).
\begin{theorem}\label{thm:asymp}
Let $1<p \leq d$, and assume that $V$ satisfies \eqref{CondC} near $\gz=0$ and that $V\in L^q_{\mathrm{loc}}(B_R)$, with $q >d/p$. Let $u$ be a positive solution of $Q'(u) =0$ in the punctured ball $B_R\sm \{0\}$.  Then either $u$ has a removable singularity at the origin, or
\begin{equation}\label{asymp:dgeqp}
u(x) \;\underset{x\to 0}{\sim}\;
\left\{
  \begin{array}{ll}
   \quad |x|^{\alpha^*} & \hbox{ if }\; p<d,\\[2mm]
    -\log |x| & \hbox{ if }\; p=d.
  \end{array}
\right.
\end{equation}
\end{theorem}
For bounded potentials the above result is due to \cite{KV,PTalt}.
The assumptions on $V$ can be further weakened to nearly optimal classes. In particular, the removability of isolated singularity  has  been studied under the assumptions that $p\leq d$, and $V$ belongs to certain (Kato) classes of functions or measures; see \cite{GZ,JV,LS,LSS,MH,P} and the references therein.

\vskip 4mm

Let us mention the other classical case. In \cite{PF} we studied the asymptotic behavior of positive $p$-harmonic functions near $\infty$ for the case $p\geq d$ (cf. \cite{Serrin2}). Using a modified Kelvin transform, we proved:
\begin{theorem}\label{thm:asymp2}
Let $p\geq d>1$, and let $u$ be a positive solution of the equation $-\pl (u) =0$ in a neighborhood of
infinity in $\mathbb{R}^d$. Then either $u$ has a removable singularity at $\infty$ (i.e., $u$ admits a finite limit as $x\to\infty$), or
\begin{equation*}
u(x) \;\underset{x\to \infty}{\sim}\;
\left\{
  \begin{array}{ll}
   \quad |x|^{\alpha^*} & \hbox{ if }\; p>d,\\[2mm]
    \log |x| & \hbox{ if }\; p=d.
  \end{array}
\right.
\end{equation*}
\end{theorem}
\begin{remark}\label{rem:p=d}{\em
The case $p=d$ in Theorem~\ref{thm:asymp2} clearly follows from Theorem~\ref{thm:asymp} using the conformality of the $d$-Laplacian.
} \end{remark}
We continue with two new general theorems concerning removable singularity. The first result deals with the {\em existence} of the limit. We prove
\begin{theorem}\label{thm:main}
Suppose that $V$ satisfies Condition~(\textbf{C1}) near $\zeta$. Let $u$ be a positive solution of the equation $Q'(u)=0$ in a punctured neighborhood of $\gz$. Then
  \begin{equation}\label{eq:ex31}
\lim_{x\to\zeta} u(x) =\ell,
\end{equation}
where $0\leq \ell\leq \infty$.

More precisely, for $p \neq d$ we have in the classical cases ($p<d$ and $\zeta = 0$, or $p>d$ and $\zeta = \infty$) that
$0<\ell\leq\infty$, and in the nonclassical cases ($p>d$ and $\zeta =0$, or $p<d$ and $\zeta = \infty$) we have $0\leq \ell < \infty$.

Assume further that $V$ satisfies Dini's condition~(\textbf{C2}) near $\gz$, then in the classical cases there are positive solutions satisfying \eqref{eq:ex31} with $0<\ell<\infty$, and with $\ell=\infty$, while in the nonclassical cases there are positive solutions satisfying \eqref{eq:ex31} with $0<\ell<\infty$, and with $\ell=0$.
\end{theorem}
\begin{remarks}\label{rem:Hardy}{\em
1. Serrin \cite{Serrin2} proved  Theorem~\ref{thm:main} for equations in divergence form without a potential term.

\vskip 3mm

2. If $p<d$ and $\gz=0$, then the finiteness of the limit implies that the
singularity is removable (see Corollary~\ref{thm:remov}).
This statement is not true for $p>d$ as can be seen from the case of the
$p$-Laplace equation and the ``fundamental solution" $v_{\alpha^*}$ (see Example~\ref{ex:pl}).

\vskip 3mm

3. The proof of Theorem~\ref{thm:main} in the nonclassical cases is based on a new general three-spheres theorem which is valid near $\gz$ (see Theorem~\ref{thm:tsi}). For related three-spheres theorems for elliptic PDEs see for example \cite{Agmon1,Brum,Landis,Mikl,PW,V} and the references therein.
}
\end{remarks}

\begin{example}{\em
\label{ex:Hardy}
The statements of Theorem~\ref{thm:main} do not hold true if $V$ has a Fuchsian type singularity at $\gz$. Indeed, consider the Hardy potential $V(x)=\gl |x|^{-p}$. So, $V$ has Fuchsian type singularities at $\zeta =0$ and  $\zeta =\infty$.
By Hardy's inequality  \eqref{Hardy_fun}, the equation
 \begin{equation}\label{eq:hardy}
    -\Delta_p(u)-\lambda \frac{|u|^{p-2}u}{|x|^p}=0
   \end{equation}
  admits a positive solution near $\gz$   (and  also in  $\mathbb{R}^d\sm \{0\}$) if and only if  $$\lambda \leq c_H:=\left|\frac{p-d}{p}\right|^p.$$
Moreover, for $\lambda = c_H$, Equation \eqref{eq:hardy} on $\mathbb{R}^d\sm \{0\}$ admits a unique (up to a multiplicative constant) positive (super)solution  which is given by $u(x)=|x|^{\gamma_*}$, where $\gamma_*:=(p-d)/p$.

On the other hand, if $\lambda < c_H$, then \eqref{eq:hardy} on $\mathbb{R}^d\sm \{0\}$ has two positive (radial) solutions of the form $v_\pm(x):=|x|^{\gamma_\pm(\lambda)}$, where
$\gamma_-(\lambda)<\gamma_*<\gamma_+(\lambda)$, and $\gamma_\pm(\lambda)$ are solutions of the transcendental 	 equation (cf. \eqref{eq:green1})
$$-\gamma |\gamma|^{p-2}[\gamma(p-1)+d-p]=\lambda.$$
In particular, the statements of Theorem~\ref{thm:main} do not hold true for \eqref{eq:hardy} with $\gl\neq 0$.
On the other hand, since Hardy's potential is radially symmetric, the ratio limit theorem holds true for positive solutions of \eqref{eq:hardy} near $\gz$ \cite{PF}.
} \end{example}

Our next result extends the results of theorems~\ref{thm:asymp} and \ref{thm:asymp2} and gives the {\em asymptotic behavior } of positive solutions near $\gz$ in the {\em nonclassical cases} ($p>d$ and $\zeta =0$, or $p<d$ and $\zeta = \infty$) under the additional assumption that the potential $V$ is {\em integrable} near the singular point $\gz$. Recall (Remark~\ref{assumptions}.) that conditions (\textbf{C1}) and (\textbf{C2}) hold under these assumptions.
\begin{theorem}\label{thm:asymp3}
Suppose that $V$ has a Fuchsian type singularity and is integrable near $\zeta$.
Assume that  $u$ is a positive solution of the equation $Q'(u) = 0$  in
the punctured ball $B_R\sm \{0\}$ and $p >d$ (respectively, $B_R^*$ and $p <d$) for some $R>0$.
Then
$$
\mbox{either }\;\;  0 < \lim_{x \to \zeta} u(x) < \infty, \quad \mbox{or }\;\; u(x) \;\underset{x\to \gz}{\sim} |x|^{\alpha^*}.
$$
\end{theorem}

\begin{remarks}{\em
1. Suppose that  $\zeta =0$ and $p>d$, then theorems~\ref{thm:main} and \ref{thm:asymp3} seem to be new even for $V\in L^\infty(B_R)$.

\vskip 3mm

2. The integrability assumption in Theorem~\ref{thm:asymp3} seems to be too restrictive; we believe that
the assertion of the theorem is still valid under assumptions (\textbf{C1}) and (\textbf{C2}).
This is indeed the case for example if $p=2 < d$ and $\zeta = \infty$ (see for example \cite{GZ,JV,LS,LSS,P}).
}
\end{remarks}

The outline of the paper is as follows. In the next section we recall
some notions and results we need throughout the paper. In particular, we present  a uniform Harnack inequality and a ratio-limit theorem for positive solutions defined
near the singularity, and we introduce Wolff's potential near $\gz$.
 In Section~\ref{sec:pos} we prove the existence of two classes of positive
solutions of the equation $Q'(u)=0$ near $\gz$, the small and the large one, with asymptotic behaviors similar
to those of the $p$-Laplace equation.
In Section~\ref{sec:small} we obtain the asymptotic behavior of small solutions near $\gz$ for the case $p \neq d$.
The proof uses Lindqvist's method in \cite{Lindqvist}.
In Section~\ref{sec:tsi} we state and prove a three-spheres theorem
that holds true for certain positive sub/supersolutions near isolated singular points. Using this three-spheres theorem and the
asymptotic behavior of small solutions, we prove in Section~\ref{sec:proofs} Theorem~\ref{thm:main}. Section~\ref{sec:asymp} is devoted to the proof of the asymptotic behavior of singular solutions (theorems~\ref{thm:asymp} and ~\ref{thm:asymp3}). In particular, Theorem~\ref{thm:asymp3} follows from the asymptotic behavior of small solutions and our three-spheres theorem. We conclude the paper in Section~\ref{sec:app} with some applications to a positive Liouville theorem in $\mathbb{R}^d$, and to the behavior near an interior isolated singularity of positive solutions of minimal growth in a neighborhood of infinity in $\Omega$.

\section{Preliminaries} \label{sec2}
The following notations and conventions will be used throughout the paper. We denote by $B_R(x_0)$ (respectively, $S_R(x_0)$) the open ball (respectively, the spheres) of radius $R$ centered at $x_0$, and let $B_R:=B_R(0)$, $S_R:=S_R(0)$, and $B_R^*:=\mathbb{R}^d\sm \overline{B_R}$.
We write $\Omega_1 \Subset \Omega_2$ if $\Omega_2$ is open, $\overline{\Omega_1}$ is
compact and $\overline{\Omega_1} \subset \Omega_2$.

For a function $u$ defined in $\{x\mid 0\leq R_1<|x|<R_2\leq \infty\}$, we use the notations
$$m_u(r):=\inf_{x \in S_r} u(x), \qquad M_u(r):=\sup_{x \in S_r} u(x) \quad \forall\; R_1<r<R_2.$$
The subscript $u$ in the notations $m_u$ and $M_u$ will be omitted when there is no danger of confusion.
Let $f,g \in C(D)$ be nonnegative functions, we denote $f\asymp g$ on
$D$ if there exists a positive constant $C$ such that
$$C^{-1}g(x)\leq f(x) \leq Cg(x) \qquad \mbox{ for all } x\in D.$$
By $f\underset{x\to \zeta}{\sim} g$ we mean that $$
\lim_{x\to \zeta}\frac{f(x)}{g(x)}= C$$ for some positive constant
$C$. Finally, for a radial function $f$ defined in a punctured neighborhood of $\gz$, we use the notation $\tilde{f}(|x|):=f(x)$.

\vskip 4mm

A function
$v\in W^{1,p}_{\mathrm{loc}}(\Omega)$ is said to be a {\em (weak) solution} of the equation $$Q'(u)=-\pl (u) + V |u|^{p-2} u = 0$$ in a domain $\Omega$ if
 \begin{equation} \label{solution} \int_\Omega (|\nabla v|^{p-2}\nabla
v\cdot\nabla\varphi+V|v|^{p-2}v\varphi)\,\mathrm{d}x=0 \qquad \forall \varphi\in\core.
\end{equation}
We say that a real function $v\in C^1_{\mathrm{loc}}(\Omega)$ is a
{\it  supersolution} (respectively, {\it  subsolution}) of
the equation $Q'(u)=0$ in $\Omega$ if for every nonnegative $\varphi\in\core$ we have
 \begin{equation}\label{supersolution}
\int_\Omega (|\nabla v|^{p-2}\nabla
v\cdot\nabla\varphi+V|v|^{p-2}v\varphi)\,\mathrm{d}x\geq 0 \;\;\mbox{ (respectively, }\leq 0\mbox{).}
\end{equation}

Recall the definition of the Wolff potential on $\mathbb{R}^d$ \cite{Martio}.  For $s > 1$, $\ga > 0$ with $0 < \ga s < d$, and a measure $\gm$, the {\em Wolff potential} is defined for $0<\gr \leq \infty$ by:
$$W^\gr_{\ga,s}\gm(x) := \int_0^{\gr} \left(\frac{\gm(B(x,t))}{t^{d-\ga s}}\right)^{1/(s-1)}
\frac{\mathrm{d}t}{t}\,.$$

Suppose that  $G: \mathbb{R}^+ \to \mathbb{R}^+$, $G(s) = g(s)|s|^{-p}$,   satisfies the Kato type Condition~(\textbf{C1}). Motivated by the above definition (with $\ga=1$, and $s=p$), we define the {\em Wolff potential of $G$ around  $\gz$}  by \begin{equation}
\label{Eq:Wolff}
W_G^\gz(x) = W_G(x) := \frac{1}{|B_1|}\left| \int_\zeta^{|x|} t^{\alpha^*} \left( \int_{X_t} G(|y|) \,\mathrm{d} y \right)^{\frac{1}{p-1}}
\frac{\mathrm{d} t}{t} \right|,
\end{equation}
where the $d$-dimensional domain $X_t$  depends on $p$, $d$, and $\zeta$ and is given by
\vskip 4mm
\begin{center}
\begin{tabular}{|c|c|c|c|}
\hline
  & $p < d$ & $p>d$ & $p=d$ \\[3mm]
\hline
$\zeta = 0$ & $X_{t} = B_t$ & $X_{t} = B_1 \sm B_t$ & $X_t =B_t$ , \\[3mm]
\hline
$\zeta = \infty$ & $X_{t} = B_t \sm B_1$ & $X_{t} = B_t^*$ & $X_t = B_t^*$.\\[3mm]
\hline
\end{tabular}
\end{center}
By integration over the spherical variables, we get the expression
\begin{equation}\label{eq:W_G2}
W_G(x) = \left| \int_\zeta^{|x|} t^{\alpha^*} \left|\int_a^t \frac{g(s)}{s^p} s^{d-1}\, \mathrm{d} s \right|^{\frac{1}{p-1}} \frac{\mathrm{d} t}{t} \right|=\left| \int_\zeta^{|x|} \left|t^{1-d}\int_a^t \frac{g(s)}{s^p} \,s^{d-1} \ds \right|^{\frac{1}{p-1}} \dt \right|< \infty,
\end{equation}
where as above,   $G(|x|) = g(|x|)|x|^{-p}$, and $a$ is given as in Condition~(\textbf{C1}) by
\vskip 4mm
\begin{center}
\begin{tabular}{|c|c|c|c|}
\hline
  & $p < d$ & $p>d$ & $p=d$ \\[3mm]
\hline
$\zeta = 0$ & $a = 0$ & $a = 1$ & $a=0$, \\[3mm]
\hline
$\zeta = \infty$ & $a = 1$ & $a = \infty$ & $a=\infty$.\\[3mm]
\hline
\end{tabular}
\end{center}
In particular, $W_G$ is well defined for $G$ in Kato's class near $\gz$.
In the sequel we use our convention $\tilde{W}_G(|x|):=W_G(x)$.
It turns out that $W_G$ solves the nonhomogeneous $p\,$-Laplace equation near $\gz$.
\begin{lemma}
\label{WolffLemma}
Suppose that $G$ satisfies Condition~(\textbf{C1}) near $\gz$. Then near $\gz$ we have
$$-\pl (W_G(x))  = \left\{
                                              \begin{array}{rl}
                                                -G(|x|) & \quad\hbox{ in the classical cases},
\\[2mm]
                                                G(|x|) & \quad \hbox{ in the nonclassical cases},
                                              \end{array}
                                            \right.
  $$
where $G(|x|)=\frac{g(|x|)}{|x|^p}$.

\vskip 2mm

Moreover, $W_G$ has the following properties
\begin{enumerate}
\item $\lim_{x \to \zeta} W_G(x) = 0$.
\item $\pd \tilde W_G(r)/\pd r > 0 $ if $\zeta =0$.
\item $\pd \tilde W_G(r)/\pd r < 0 $ if $\zeta = \infty$.
\end{enumerate}
Furthermore,  in the nonclassical cases, if $G$ is integrable, then
$$
W_G(x) \;\underset{x\to \gz}{\sim}\; v_{\alpha^*}(x).
$$
\end{lemma}
\begin{proof}
The first statement is a  straightforward computation (cf. \cite{Qi}).
By Lebesgue dominated convergence theorem we have $\lim_{x \to \zeta} W_G(x) = 0$. Assertions \textit{2} and \textit{3} are obtained by differentiation.
Finally, the last assertion of the lemma follows from \eqref{eq:W_G2}  by l'H\^{o}pital's rule
and the integrability assumption.
\end{proof}

We introduce another type of potential for functions $G$ satisfying (\textbf{C2}).
\begin{equation}\label{UpotentialY}
U_G(x):=
\left\{ \begin{array}{lr}
\displaystyle{\alpha^* \int_{b}^{v_{\alpha^*}(x)} \int_\gz^{v_{\alpha^*}^{-1}(\gt)} \frac{g(s)}{s} \,\mathrm{d} s \,\mathrm{d} \tau} & p \neq d, \\[6mm]
\displaystyle{\int_{1}^{v_{\alpha^*}(x)} \int_\gz^{v_{\alpha^*}^{-1}(\gt)} \frac{g(s)}{s} |\log s|^{\,d-1} \,\mathrm{d} s \,\mathrm{d} \tau} & p =d,
\end{array} \right.
\end{equation}
where $b$ is given by
\vskip 4mm
\begin{center}
\begin{tabular}{|c|c|c|}
\hline
  & $p < d$ & $p>d$ \\[3mm]
\hline
$\zeta = 0$ & $b = 1$ & $b = 0$, \\[3mm]
\hline
$\zeta = \infty$ & $b = 0$ & $b = 1$.\\[3mm]
\hline
\end{tabular}
\end{center}
We give up of the positivity of $U_G$, it plays no role in the sequel.
It can be verified easily that $U_G$ satisfies the following properties.
\begin{lemma}\label{Ulemma}
Let $G$ satisfy (\textbf{C2}), then $U_G$ is well defined and
$$
\lim_{x \to \zeta} \frac{U_G(x)}{v_{\alpha^*}(x)}
= \lim_{x \to \zeta} \pd_{v_{\alpha^*}} U_G(x) = 0.
$$
Moreover,
$$
\left(\pd^2_{v_{\alpha^*}} U_G\right)(x) = \left\{
                                             \begin{array}{ll}
                                               \dfrac{g(|x|)}{v_{\alpha^*}(|x|)} & \quad p\neq d ,\\[6mm]
                                               \dfrac{g(|x|)}{(v_{\alpha^*}(|x|))^{1-d}} & \quad  p=d.
                                             \end{array}
                                           \right.
$$
\end{lemma}
\begin{proof}
The second statement follows by a direct computation, while the first follows by l'H\^{o}pital's rule  and
Lebesgue's dominated convergence theorem.
\end{proof}

We are now ready to prove Lemma~\ref{C1C@weakFuchs} claiming that each of our assumptions (\textbf{C1}) and (\textbf{C2}) implies that $V$ has a weak Fuchsian singularity at $\gz$.

\begin{proof}[Proof of Lemma~\ref{C1C@weakFuchs}]
We first prove the claim assuming (\textbf{C2}). Note that if $p=d$,  then (\textbf{C2}) implies that
$$\left|\int_\gz^1 \frac{g(s)}{s} \,\mathrm{d}s\right|<\infty.$$
Let $R_n\to \gz$. It is sufficient to prove that
$$|R_n|^p\frac{g(R_n|x|)}{|R_n|^p|x|^p}= \frac{g(R_n|x|)}{|x|^p}\underset{n\to \infty}{\longrightarrow} 0 \quad \mbox{ in  the weak$^*$ topology of  $L_{\mathrm{loc}}^{\infty}(\mathbb{R}^d\sm \{0\})$}.$$
So, it suffices  to prove that $g(R_n s) \to 0$ in the weak* topology
of $L^\infty_{\mathrm{loc}}(\mathbb{R}_+\sm \{0\})$.

By compactness, there is $g_\infty\in L^\infty_{\mathrm{loc}}(\mathbb{R}_+\sm \{0\})$
such that (up to a subsequence) $$g(R_n x)\;  \underset{n\to \infty}{\longrightarrow} \; g_\infty(x)$$ in the weak* topology
of $L^\infty_{\mathrm{loc}}(\mathbb{R}_+\sm \{0\})$.
So for any fixed $0<a<b<\infty$ we have
$$
\int\limits_a^b \frac{g(R_n s)}{s} \,\mathrm{d} s \;\underset{n\to \infty}{\longrightarrow}\; \int\limits_a^b \frac{g_\infty(s)}{s} \,\mathrm{d}s.
$$
On the other hand,
$$
\int\limits_a^b \frac{g(R_n s)}{s} \ds = \int\limits_{R_n a}^{R_n b} \frac{g(s)}{s} \ds \;\underset{n\to \infty}{\longrightarrow}\;  0
$$
by Lebesgue dominated convergence theorem. Hence $g_\infty = 0$.

Assume now that (\textbf{C1}) is satisfied. The structure of the proof is similar. Let $G_n(x):=R_n^pG(R_n x)$. We may suppose that
$$G_n(x) \;\underset{n\to \infty}{\longrightarrow}\; G_\infty(x)$$ in the weak* topology
of $L^\infty_{\mathrm{loc}}(\mathbb{R}_+\sm \{0\})$.
By Fatou's lemma we have
$$
W_{G_\infty}(x) \leq \liminf_{n \to \infty} W_{G_n}(x).
$$
On the other hand, by a direct computation we obtain
\begin{equation}\label{eq:wgn}
W_{G_n}(x) = \frac{1}{|B_1|} \left| \int_{\zeta}^{R_n |x|} t^{\alpha^*} \left( \int_{R_n X_{(t/R_n)}} G(|y|) \mathrm{d}y
\right)^{\frac{1}{p-1}} \frac{\mathrm{d} t}{t}\right|.
   \end{equation}
We note that (both in the classical cases and in the nonclassical cases) if  $x$ is near $\gz$ and  $t$ belongs to the interval of the integration in \eqref{eq:wgn}, then $R_n X_{(t/R_n)} \subset X_t$.
Consequently, for such $x$ we have
$$
W_{G_n}(x) \leq \frac{1}{|B_1|} \left| \int_{\zeta}^{R_n |x|} t^{\alpha^*} \left( \int_{X_{t}} G(|y|) \mathrm{d}y
\right)^{\frac{1}{p-1}} \frac{\mathrm{d} t}{t}\right|.
$$
 Therefore, our assumption (\textbf{C1}) and
Lebesgue dominated convergence theorem imply that $W_{G_n}(x) \to 0$. Hence,  we arrived at
$$
W_{G_\infty}(x) = 0
$$
near $\zeta$. Thus, $G_\infty(x) = 0$ near $\zeta$.
\end{proof}
Next, we present a simple sufficient condition ensuring the existence of positive solutions near $\gz$. We note that  Lemma~\ref{lem:green} implies that such positive solutions exist (for any $1<p<\infty$) if either condition (\textbf{C1}) or (\textbf{C2}) are satisfied.
\begin{lemma}\label{lem:pos}
Suppose that $p\neq d$ and that $V$ satisfies \eqref{CondC}  with  a continuous positive function $g$ satisfying
 $\;\lim_{x\to\gz}g(|x|)=0$. Then the equation $Q'(u)=0$ admits positive solutions in a punctured neighborhood of the singular point $\zeta$.
\end{lemma}
\begin{proof}
For $\zeta=0$ (respectively, $\zeta=\infty$) and $p\neq d$, the lemma's assumption and Hardy's inequality
\begin{equation}\label{Hardy_fun}
    \int_{\mathbb{R}^d\setminus \{0\}} |\nabla u|^p \,\mathrm{d}x\geq
  \left|\frac{p-d}{p}\right|^p \int_{\mathbb{R}^d\setminus \{0\}} \frac{|u|^p}{|x|^p}\,\mathrm{d}x
 \qquad u\in C_0^\infty(\mathbb{R}^d\setminus \{0\}),
\end{equation}
imply that for $R$ small (respectively, large) enough the functional 
\begin{equation*}\label{Q3} Q(u):=\int_{\Gw}\left(|\nabla u|^p+V|u|^p\right)\,\mathrm{d}x
\end{equation*}
is nonnegative on $C_0^\infty(\Gw)$, where $\Gw=B_R\sm \{0\}$ (respectively, $\Gw=B_R^*$), and therefore, by \cite{PTalt}, the equation $Q'(u)=0$ admits positive  solutions in $B_R\sm \{0\}$ (respectively, $B_R^*$).

\end{proof}
We recall that positive super- and subsolutions satisfy the following weak comparison principle.
\begin{theorem}[Weak comparison principle \cite{Garcia}]\label{thm:wcp}
Let $V\in  L^\infty_{\mathrm{loc}}(\Omega)$, and let $\Omega'$ be a bounded $C^{1,\alpha}$ subdomain of a domain $\Omega$, such that $\Omega'\Subset \Omega$.
Assume that the equation $Q'(w)=0$ admits a positive solution in $\Omega$ and
suppose that $u,v\in C^1(\Omega')\cap C(\overline{\Omega'})$, $u,v\geq 0$ satisfy the following inequalities
\begin{eqnarray}\label{eq:wcp}
      Q'(u) \leq 0  & \mbox{in} & \Omega', \nonumber\\
      Q'(v) \geq 0 & \mbox{in} & \Omega', \\
 u \leq v & \mbox{on} & \bo \Omega' \nonumber.
\end{eqnarray}
Then $u \leq v$ in $\Omega'$.
\end{theorem}
Our results hinges on a uniform Harnack estimates for positive solutions near an isolated Fuchsian-type singularity \cite{veron,PT}, and on a ratio-limit theorem for the quotients of any such two solutions near
a weak-Fuchsian singularity \cite[Theorem 2.6]{PF}. The following statements
are formulated for the potentials under consideration.
\begin{proposition}\label{prop:harnack}
Assume that $V$ has a Fuchsian type singularity at $\gz$. Fix $0<r_0<R$ (respectively, $r_0>R$).
Then there exists a constant $C=C(r_0,R,V,d,p)>0$ such that any  positive solution $u$ of the equation $Q'(u)=0$ in $B_R \sm \{0\}$ (respectively, $B_R^*$) satisfies the following uniform Harnack inequality
\begin{equation}
M_u(r) \leq C  m_u(r) \quad \mbox{ for all} \quad 0<r<r_0 \; \;\mbox{(respectively, $r>r_0$)}.
\label{eq:harnack}
\end{equation}
\end{proposition}
The above uniform Harnack inequality together with a dilatation argument and a standard elliptic estimates imply the following gradient estimates:
\begin{lemma}
\label{thm:grad}
Assume that $V$ has a Fuchsian type singularity at $\gz$. Fix $0<r_0<R$ (respectively, $r_0>R$). Then there exists a constant $C=C(r_0,R,V,d,p)>0$ such that any  positive solution $u$ of the equation $Q'(u)=0$ in $B_R \sm \{0\}$ (respectively, $B_R^*$) we have
\begin{equation}\label{grad}
|\grad u(x)| \leq C\frac{u(x)}{|x|} \quad \mbox{ for all} \quad 0<|x|<r_0 \; \;\mbox{(respectively, $|x|>r_0$)}.
\end{equation}
\end{lemma}
Lemma~\ref{thm:grad} implies the following removable singularity result.
\begin{corollary}
\label{thm:remov}
Suppose that $V$ has a Fuchsian type singularity at the origin, and that  $p<d$. Let $u$ be a positive bounded solution of the equation $Q'(u)=0$ in the punctured ball $B_R\sm\{0\}$.  Then $u$ has a removable singularity at the origin.
\end{corollary}
\begin{proof}
By Lemma~\ref{thm:grad} $u\in W^{1,\,p}(B_R)$, and consequently $u$ is a positive solution in the ball $B_R$.
\end{proof}

In \cite{PF} we proved the following ratio limit theorem.
\begin{proposition}[{\cite[Theorem 2.6]{PF}}]
\label{thm:reg}
Assume that $V$ has a weak Fuchsian type singularity at $\gz$.  Let $u$ and $v$ be two positive solutions of the equation $Q'(w) = 0$ in a punctured neighborhood of $\gz$. Then the limit
$$
\lim_{x \to \zeta} \frac{u(x)}{v(x)} \quad \mbox{exists;}
$$
the limit might be infinite.
\end{proposition}
We conclude this section with the following elementary lemma whose proof can be easily verified.
\begin{lemma}\label{lem:green1}
Let $u$ be a positive solution of the equation $Q'_V(u)=0$ in $\Omega$, and $f\in C^2(\mathbb{R}_+)$ be a positive function. Then
\begin{equation}\label{eq:Qf}
Q'_V(f(u))=-(p-1)|f'(u)|^{p-2}f''(u)|\nabla u|^p-|f'(u)|^{p-2}f'(u)V|u|^{p-1}+ Vf(u)^{p-1}.
\end{equation}
\end{lemma}
\section{Existence of the desired solutions}\label{sec:pos}
In this section we prove the existence of positive
subsolutions, supersolutions, and  solutions of the equation $Q'(u)=0$ near $\gz$ with the desired asymptotic behavior.
We present explicit formulas for such sub and supersolutions in terms of the potentials $W_G$ and $U_G$ (see \eqref{Eq:Wolff} and \eqref{UpotentialY}). The existence of the appropriate solutions follows using Perron's method (cf. \cite{Garcia}, in particular the discussion around (3.1) therein).
\begin{lemma}\label{lem:green}
Suppose that $V$ satisfies Condition~(\textbf{C1}) (respectively,  (\textbf{C2})) near $\zeta$. Denote
\begin{equation}\label{upmvpm}
    u_\pm(x) :=\left\{
      \begin{array}{ll}
        1 \mp C W_G(x) & \hbox{ in the classical cases,} \\[2mm]
        1\pm C W_G(x) & \hbox{ in the nonclassical cases.}
      \end{array}
    \right.
\end{equation}
$$
  v_\pm(x) := v_{\alpha^*}(x)  \mp C U_G(x),
$$
where $W_G$, $U_G$ and $v_{\alpha^*}(x)$ are defined by \eqref{Eq:Wolff}, \eqref{UpotentialY} and \eqref{eq:free}, respectively, and $C$ is a large positive constant. Then $u_+$ (respectively, $v_+$) is  a positive supersolution and $u_-$ (respectively, $v_-$) is a positive subsolution of the equation $Q'(u)=0$ near the singular point $\zeta$.
Moreover,
$$ u_\pm(x) \;\underset{x\to \gz}{\sim}\; 1, \qquad \big(\mbox{respectively, } \; v_\pm(x) \;\underset{x\to \gz}{\sim}\; v_{\alpha^*}(x) \big).$$
\end{lemma}
\begin{proof}
In light of  Lemma~\ref{WolffLemma} we have that   $u_\pm \sim 1$ near $\zeta$, and $-\pl u_\pm = \pm C^{p-1}G(|x|)$.
Therefore, near $\gz$ we have
\begin{equation}\label{eq:Qf1}
    Q'_V(u_\pm) = \pm C^{p-1}G(|x|) + u_\pm^{p-1} V(x)\; \substack{\geq \\ \leq}  \;\pm C^{p-1}G(|x|) \mp (1+\vge)|V(x)|\; \substack{\geq \\ \leq} \;0.
  \end{equation}

Similarly, Lemma~\ref{Ulemma} implies that $v_\pm \sim v_{\alpha^*}$ near $\zeta$. Using
Lemma~\ref{lem:green1} and Lemma~\ref{Ulemma},  we obtain
$$ -\pl v_\pm \sim \pm C^{p-1} \left(\pd^2_{v_{\alpha^*}} U_G\right)(x) |\grad v_{\alpha^*}(x)|^{p}
\sim \pm C^{p-1}\frac{g(|x|)}{|x|^p}\,v_{\alpha^*}^{p-1}.
$$
Therefore, near $\gz$ we have
\begin{multline}\label{eq:Qf2}
    Q'_V(v_\pm) \sim \pm C^{p-1}G(|x|) v_{\alpha^*}^{p-1}(x) + V(x) v_{\alpha^*}^{p-1}(x) \;\substack{\geq \\ \leq} \;\\[2mm]
\big(\pm C^{p-1}G(|x|) \mp (1+\vge)|V(x)|\big)v_{\alpha^*}^{p-1}(x)\; \substack{\geq \\ \leq} \;0.
  \end{multline}
\end{proof}
\begin{remark} \label{rem:p=d1} {\em
The existence of positive supersolutions of the equation $Q'(u)=0$ in a domain $D$ implies the existence of positive solutions of the equation $Q'(u)=0$ in $D$. Hence,  Lemma~\ref{lem:green} provides us with a proof of the claim of Lemma~\ref{lem:pos} under different assumptions, and in particular, covers the missing case $p=d$.
} \end{remark}
We conclude this section with a lemma that claims that as in the case of the $p\,$-Laplacian (see Example~\ref{ex:pl}), there exist two positive solutions $u_{\mathrm{small}}$ and $u_{\mathrm{large}}$  of the equation $Q'(u)=0$   near $\gz$, each of which behaves asymptotically either as the constant function or as the ``fundamental solution" $v_{\alpha^*}(x)$ (see \eqref{eq:free}) of the $p\,$-Laplacian.
\begin{lemma}\label{lem:0sol}
  Suppose that $V$ satisfies conditions~(\textbf{C1}) and (\textbf{C2}) near $\zeta$.
 Then there exist two positive solutions $u_{\mathrm{small}}^{(\gz)}$ and $u_{\mathrm{large}}^{(\gz)}$ of the equation $Q'(u)=0$ defined in a punctured neighborhood of $\gz$ and satisfying
  $$\lim_{x \to \zeta} \frac{u_{\mathrm{small}}^{(\gz)}(x)}{u_{\mathrm{large}}^{(\gz)}(x)}=0 .$$
   More precisely,  we have
   \begin{equation*}
 u_{\mathrm{small}}^{(\gz)}(x) \,\underset{x\to \gz}{\sim} \,\min\{
 1 ,   v_{\alpha^*}(x)\},
\end{equation*}
and
\begin{equation*}
 u_{\mathrm{large}}^{(\gz)}(x) \,\underset{x\to \gz}{\sim} \,\max\{
 1 ,   v_{\alpha^*}(x) \}.
\end{equation*}
\end{lemma}
\begin{proof}
By Lemma~\ref{lem:green} we have positive subsolutions and supersolutions which behave as the desired solutions. Therefore, using a standard sub-supersolution argument (based on the weak comparison principle, cf. \cite{Garcia}, in particular the discussion around (3.1) therein), it follows that there exist positive solutions with the desired asymptotic behavior.
 \end{proof}

\section{The asymptotic of `small' solutions near $\gz$}
\label{sec:small}
Throughout this section we assume that $p \neq d$.  The aim of this section is to prove that all  positive solutions of the equation
$Q'(u) = 0$ near $\gz$ that are bounded by the small positive solution of the $p$-Laplace equation near $\gz$, have the same asymptotic behavior.  Hence, if $V$ satisfies conditions~(\textbf{C1}) and (\textbf{C2}) near $\zeta$, and if a solution $u$ near $\gz$ satisfies  $u \leq Cu_{\mathrm{small}}^{(\gz)}$, then $u \sim u_{\mathrm{small}}^{(\gz)}$.
\begin{theorem}\label{lindqvist}
 Suppose that $V$  has a weak Fuchsian type singularity at $\zeta$.
 Let $v_1,\,v_2$ be two distinct positive solutions of the equation
$Q'(u) = 0$ in a punctured neighborhood $\Gw$ of $\gz$  satisfying
$$v_i(x) \leq C\min\{1, |x|^{\alpha^*}\} \qquad \mbox{near } \gz, $$
where  $C$ is a positive constant and $i=1,2$.
Then $$v_1 \;\underset{x\to \gz}{\sim}\;  v_2.$$
\end{theorem}
\begin{proof}
The proof is based on Lindqvist's method \cite{Lindqvist}.
Without loss of generality we assume that $\pd \Omega\sm \{\gz\}=S_R$ for some $R>0$, and  $v_i\in C(\bar{\Gw}\sm \{\gz\}),\, v_i > 0,\, i=1,2$, on $S_R$. Then for some $C_1>0$ small $C_1 v_i \leq 1 \leq C_1^{-1} v_i$ on $S_R$. By Perron's sub-supersolution (cf. \cite{Garcia}, in particular the discussion around (3.1) therein)
method, there exist positive solutions $u_i$ of the Dirichlet problem
\begin{align*}
Q'(u_i) = 0 \qquad \mbox{in } \Gw,\\
u_i = 1 \qquad \mbox{on } S_R, \\
u_i \asymp v_i \qquad \mbox{near } \zeta.
\end{align*}
Due to the existence of the ratio limit near $\zeta$, we have $u_i \sim v_i$ near $\zeta$.
 Moreover, using an elementary comparison argument we get that either $u_1 \geq u_2$ or $u_2 \geq u_1$. So, without loss of generality,  we assume that $u_1 \geq u_2$.
We claim that $u_1 = u_2$ and this implies the lemma.

We use Lindqvist technique \cite{Lindqvist}. For completeness, we give a self-contained proof.
 Nevertheless, the reader is advised to consult the above paper
for details.

By our assumption $u_i(x) \leq C\min\{1, |x|^{\alpha^*}\}$, and therefore,
\begin{equation}
\label{lin1}
\int_{\Omega} u_i^p \frac{1}{|x|^p} \,\mathrm{d} x < \infty.
\end{equation}
On the other hand, Lemma~\ref{thm:grad} implies that  $|\grad u_i(x)| \leq C_2 u_i (x)|x|^{-1}$, and hence
$\grad u_i\in L^p(\Omega)$. Consequently,  for $0<\vge<1$, the $C^{1,\ga}$-function
$$
\mu_\vge := \frac{(u_1 + \vge)^p - (u_2 + \vge)^p}{(u_1 + \vge)^{p-1}}
$$
that vanishes on $S_R$ is a valid test function for the equation $Q'(u) = 0$.
In particular,
$$
\int_\Omega |\grad u_1|^{p-2} \grad u_1 \cdot \grad \mu_\vge \mathrm{d} x
 + \int_{\Omega} V u_1^{p-1} \mu_\vge  \mathrm{d} x = 0.
$$
Adding together this equation and the one with $u_1,\,u_2$ interchanged, we obtain
\begin{multline}
\label{lin2}
-\int_{\Omega}V(x)\left[ \frac{u_1^{p-1}}{(u_1 + \vge)^{p-1}} - \frac{u_2^{p-1}}{(u_2 + \vge)^{p-1}} \right]
\bigg((u_1 + \vge)^p - (u_2 + \vge)^p\bigg) \mathrm{d} x = \\[3mm]
 \!\!\int_\Omega \!\!(u_1 + \vge)^p L(\grad\log(u_1 + \vge),\grad\log(u_2 + \vge))\mathrm{d} x\!
+\!\! \int_\Omega \!\!(u_2 + \vge)^p L(\grad\log(u_2 + \vge),\grad\log(u_1+ \vge))\mathrm{d} x,
\end{multline}
where (see \cite{Lindqvist} or \cite{PTalt})
\begin{equation}\label{abineq}
L(a,b) := |a|^p - |b|^p - p |b|^{p-2} b \cdot (a - b) \geq 0 \qquad \forall a,b\in \mathbb{R}^d
\end{equation}
and equality occurs in \eqref{abineq} if and only if $a=b$. In particular, (recalling that $u_1=u_2=1$ on $S_R$) the right hand side of \eqref{lin2} is zero if and only if $u_1=u_2$

 On the other hand, the integrand of the left hand side of \eqref{lin2} tends to zero as $\vge\to 0$.
Moreover, due to Lagrange theorem (recall that $u_1 \geq u_2$) we have for $0<\vge < 1$
\begin{multline*}
|V(x)|\left| \left[ \frac{u_1^{p-1}}{(u_1 + \vge)^{p-1}} - \frac{u_2^{p-1}}{(u_2 + \vge)^{p-1}}\right]
\bigg((u_1 + \vge)^p - (u_2 + \vge)^p\bigg) \right|\leq\\[2mm]
|V(x)| \frac{u_1^{p-1}}{(u_1 + \vge)^{p-1}}\,p(u_1- u_2)(u_1 + \vge)^{p-1} \leq  p|V(x)|u_1^{p}\in L^1(\Gw).
\end{multline*}
Therefore, by letting $\vge\to 0$ in \eqref{lin2} and using Lebesgue dominated convergence theorem and Fatou Lemma,  we arrive at the desired
conclusion that
$ u_1 = u_2$.
\end{proof}

The uniqueness of the Dirichlet problem for small solutions is an immediate
corollary.

\begin{corollary}
Suppose that $V$ has a weak Fuchsian type singularity at $\zeta$, and let
$\Omega, \Gw'$ be punctured neighborhoods of $\zeta$ such that  $\partial \Gw\sm \{\gz\}$ is smooth, $\bar{\Gw}\sm \{\gz\}\subset \Gw'$, and $Q\geq 0$ on $C_0^\infty(\Gw')$. Then for any positive $\phi\in C(\partial \Gw\sm \{\gz\})$
the following Dirichlet problem
\begin{align*}
Q'(u)& = 0 \qquad \mbox{in } \Gw,\\
u &= \phi \qquad \mbox{on } \partial \Gw\sm \{\gz\}, \\
u(x) &\leq Cu_{\mathrm{small}}^{(\gz)}(x)  \qquad \mbox{near } \zeta.
\end{align*}
has at most one solution.
\end{corollary}

\section{Three-spheres theorem}
\label{sec:tsi}
In this section we prove a three-spheres theorem for certain positive super- and subsolutions $u$ of the equation $Q'_V(v)=0$ near an isolated singular point $\gz$.
Recall that for a fixed function $u$ and $r>0$, we denote
$$m(r):=\inf_{x \in S_r} u(x), \qquad M(r):=\sup_{x \in S_r} u(x). $$ The classical Hadamard three-spheres theorem states
that for any positive subharmonic function $u$ in an open ball $B_R\subset \mathbb{R}^d$, $d>2$ (respectively, $d=2$),  we have that $M(r)$ is a convex function of $r^{2-d}$ (respectively, $\log r$) (see for example \cite{PW}).

Suppose that $V$ satisfies Condition~(\textbf{C1}) near $\zeta$, and let $\tilde{W}_{G}=W_G^\gz$ be the Wolff potential around $\gz$.
We show that if $u$ is a certain  positive sub/supersolution of the equation $Q'(v)=0$ near $\gz$, then  the function  $m_u(r)$ (respectively,  $M_u(r)$) is a concave (respectively, convex) function of $\tilde{W}_G$ {\em near} the isolated singular point $\gz$.
Let
\begin{equation}\label{eq:deff}
f(|x|):=C_1 + C_2 \tilde{W}_{G}(|x|),
\end{equation}
 where $C_1, C_2$ depend on $u$.
  The proof is based on Lemma~\ref{lem:green} that claims that for certain values of $C_1$, $C_2$, the function $f(|x|)$ is either a super- or subsolution near the singular point $\zeta$.

\begin{theorem}
\label{thm:tsi}
Assume that $V$ satisfies Condition~(\textbf{C1}), and let $W:=W_G^\gz$ be the corresponding Wolff potential around $\gz$.

 Suppose that $u$ is a positive subsolution of the equation $Q'(v)=0$ near a punctured neighborhood of $\gz$, satisfying one of the following conditions:
\vskip 4mm
\begin{center}
\begin{tabular}{|c|c|c|}
  \hline
  1.1& If $\gz=0$  and $p>d$, & then $\lim_{x\to 0} u(x)=0$.  \\[3mm]\hline
  1.2& If $\gz=0$ and $p\leq d$, & then $\lim_{x\to 0} u(x)=\infty$.  \\[3mm]\hline
 1.3& If $\gz=\infty$ and $p\geq d$, & then $\lim_{x\to\infty} u(x)=\infty$.  \\[3mm]\hline
 1.4& If  $\gz=\infty$ and  $p<d$, & then $\lim_{x\to \infty} u(x)=0$.  \\[3mm]
  \hline
\end{tabular}
\end{center}
Then for the case  $\gz=0$ (respectively, $\gz=\infty$) there is $0<R_2$ such that for every $r_2< R_2$  (respectively, $R_2<r_2$) there exists $0<R_1(r_2)<r_2$ (respectively, $0<r_2<R_1(r_2)$), such that for every $0<r_1 < R_1(r_2)< r_2 < R_2$ (respectively, $0<R_2<r_2 < R_1(r_2)< r_1$) the following convex three-spheres inequality
\begin{equation}
\label{eq:tsiM}
 M(r_3)\leq M(r_1) \frac{\tilde{W}(r_2) - \tilde{W}(r_3)}{\tilde{W}(r_2) - \tilde{W}(r_1)}
+M(r_2) \frac{\tilde{W}(r_3)-\tilde{W}(r_1)}{\tilde{W}(r_2) - \tilde{W}(r_1)}
 \end{equation}
holds true for all $r_3 \in (r_1, r_2)$ (respectively, $r_3 \in (r_2,r_1)$). Moreover, $M$ is monotone near $\gz$.

Similarly, suppose that $u$ is a positive supersolution of the equation $Q'(v)=0$ near a punctured neighborhood of $\gz$, satisfying one of the following conditions:
\vskip 4mm
\begin{center}
\begin{tabular}{|c|c|c|}
  \hline
  2.1& If $\gz=0$  and $p>d$, & then $\lim_{x\to 0} u(x)=\infty$.  \\[3mm]\hline
  2.2& If $\gz=0$ and $p\leq d$, & then $\lim_{x\to 0} u(x)=0$.  \\[3mm]\hline
 2.3& If $\gz=\infty$ and $p\geq d$, & then $\lim_{x\to\infty} u(x)=0$.  \\[3mm]\hline
 2.4& If  $\gz=\infty$ and  $p<d$, & then $\lim_{x\to \infty} u(x)=\infty$.  \\[3mm]
  \hline
\end{tabular}
\end{center}
Then for the case  $\gz=0$ (respectively, $\gz=\infty$) there is $0<R_2$ such that for every $r_2< R_2$  (respectively, $R_2<r_2$) there exists $0<R_1(r_2)<r_2$ (respectively, $0<r_2<R_1(r_2)$), such that for every $0<r_1 < R_1(r_2)< r_2 < R_2$ (respectively,  $0<R_2<r_2 < R_1(r_2)< r_1$) the following concave three-spheres inequality holds true
\begin{equation}
\label{eq:tsim}
  m(r_1) \frac{\tilde{W}(r_2) - \tilde{W}(r_3)}{\tilde{W}(r_2) - \tilde{W}(r_1)}
+m(r_2) \frac{\tilde{W}(r_3)-\tilde{W}(r_1)}{\tilde{W}(r_2) - \tilde{W}(r_1)} \leq m(r_3)
 \end{equation}
holds true for all $r_3 \in (r_1, r_2)$ (respectively, $r_3 \in (r_2,r_1)$).
Moreover, $m$ is monotone near $\gz$.
\end{theorem}

\begin{proof}
We will only prove the case \textit{2.1} of the theorem, namely, the concave three-spheres inequality \eqref{eq:tsim} for $m(r)$ under the assumption $p>d$ and $u \to \infty$ as $x\to 0$ (see Remark~\ref{rem:conc}). The proofs for the other cases are similar.

\vskip 4mm

Since $\lim_{x\to 0}u(x)= \infty$, it follows that for every $0<r_2<R$ (to be determined later)
there exists $R_1:=R_1(r_2)<r_2$ such that
\begin{equation}
\label{eq:assumU}
m(r_1) > 2 m(r_2)\qquad \forall\, 0<r_1<R_1.
\end{equation}

  For $0<r_2<R$ and  $0<r_1<R_1(r_2)$, define an auxiliary function
$$f :  \{x\mid r_1 < |x| <r_2\}\to \mathbb{R}_+,$$
\begin{multline}\label{eq:defw}
f(x) := \tilde{f}(|x|) := m(r_1) \frac{\tilde{W}(r_2) - \tilde{W}(|x|)}{\tilde{W}(r_2) - \tilde{W}(r_1)}
+m(r_2) \frac{\tilde{W}(|x|)-\tilde{W}(r_1)}{\tilde{W}(r_2) - \tilde{W}(r_1)}
= \\[4mm]
 \frac{m(r_1)\tilde{W}(r_2) - m(r_2)\tilde{W}(r_1)}{\tilde{W}(r_2) - \tilde{W}(r_1)}
+ \frac{m(r_2)-m(r_1)}{\tilde{W}(r_2) - \tilde{W}(r_1)}\tilde{W}(|x|)=C_1+C_2\tilde{W}(|x|).
\end{multline}
Note that assumption \textit{2.1} and $\tilde{W}(r_2) > \tilde{W}(r_1)$ (Lemma \ref{WolffLemma}) imply
 that $C_2<0$.
Therefore, $f$ is a positive monotone decreasing function of $|x|$ satisfying $f(r_i) = m(r_i),\,i=1,2$.
In particular,
\begin{equation}
f(x) \leq m(r_1) \qquad \forall\, r_1 \leq |x| \leq r_2.
\end{equation}
Note that $f$ is of the form of $u_-$ of Lemma~\ref{lem:green}, therefore, for the case \textit{2.1} $f$ should be a positive subsolution near $\gz=0$. Indeed,
using Lemma~\ref{WolffLemma} we obtain
\begin{multline}\label{eq:Qf8}
    Q'_V(f)=-\pl f + V(x) (f(|x|))^{p-1}\leq
      -\left(\frac{m(r_1) - m(r_2)}{\tilde{W}(r_2) - \tilde{W}(r_1)} \right)^{p-1} G(|x|) + G(|x|)f(|x|)^{p-1} \leq \\[4mm]
      -\left(\frac{m(r_1)}{2\tilde{W}(R)}\right)^{p-1} G(|x|)
      +\left(m(r_1)\right)^{p-1} G(|x|).
\end{multline}
  It follows (recall that $\tilde{W}(r) \to 0$ as $r \to 0$) that there exists
  $0<R=:R_2$ such that for every $r_2< R_2$  there exists $0<R_1(r_2)<r_2$, such that for every $0<r_1 < R_1(r_2)< r_2 < R_2$  the function
$f$ is a subsolution in $\{x\mid r_1 < |x| < r_2\}$.

Consequently, for such $r_1$ and $r_2$, the weak comparison principle implies that
\begin{equation}\label{eq:uleqw}
m(r_1) \frac{\tilde{W}(r_2) - \tilde{W}(|x|)}{\tilde{W}(r_2) - \tilde{W}(r_1)}
+m(r_2) \frac{\tilde{W}(|x|)-\tilde{W}(r_1)}{\tilde{W}(r_2) - \tilde{W}(r_1)}
 =f(|x|) \leq u(x) \qquad \mbox{in } \{x\mid r_1 < |x| < r_2\}.
\end{equation}
Finally, for $r_3 \in(r_1, r_2)$ we take in \eqref{eq:uleqw} the infimum over the spheres $S_{r_3}$,  and we arrive at the desired three-spheres inequality \eqref{eq:tsim}.

It remains to prove the monotonicity of $m$ as $r\to 0^+$. Indeed, using \eqref{eq:tsim}, it follows that for  $r_2<R_2$ and $0<r_1<R_1(r_2)$ we have \begin{equation}
\label{eq:minineq}
\min(m(r_1), m(r_2))
\leq m(r_3) \qquad \forall \;r_3 \in(r_1, r_2).
\end{equation}
Since by our assumption $\lim_{r\to 0^+} m(r)=\infty$, \eqref{eq:minineq} clearly implies that $m$ is a decreasing
function of $r$ near $0$.
\end{proof}

\begin{remarks}\label{rem:conc}{\em
1. According to theorems~\ref{thm:asymp} and \ref{thm:main}, positive solutions of the equation $Q'(u)=0$ in a punctured neighborhood of $\gz$ never satisfy one of the conditions \textit{2.1--2.4} of Theorem~\ref{thm:tsi}. Nevertheless, the validity of the concave three-spheres inequality \eqref{eq:tsim} in these cases is  an essential part of the proof of Theorem~\ref{thm:main}. Therefore, we chose to prove one of these cases in detail.

2. Theorem~\ref{thm:tsi} can be strengthened by replacing the assumptions $\lim_{x\to\gz} u(x)=0,\infty$, by the weaker assumptions $\liminf_{x\to\gz} u(x)=0$ and $\limsup_{x\to\gz} u(x)=\infty$, respectively. We shall not elaborate this point in the present paper.
}
\end{remarks}

\section{Proof of Theorem~\ref{thm:main}}\label{sec:proofs}
We first prove the following lemma.
\begin{lemma}\label{lem:lim}
Suppose that $V$ satisfies Condition~(\textbf{C1}) near $\zeta$. Let $u$ be a positive solution of the equation $Q'(u)=0$ in a punctured neighborhood of $\gz$. Then
  \begin{equation}\label{eq:ex3}
\lim_{x\to\zeta} u(x) \quad \mbox{exists;}
\end{equation}
the limit might be infinite.
\end{lemma}
\begin{proof}
 By Lemma~\ref{lem:0sol},  the equation $Q'(w)=0$ admits a positive solution $v$ in a punctured neighborhood of $\gz$  satisfying
\begin{equation}\label{eq:lim1}
\lim_{x \to \gz} v(x)=1.
\end{equation}
Let $u$ be a positive solutions of the equation $Q'(w)=0$ in a such punctured neighborhood. By Proposition~\ref{thm:reg},
$$
\lim_{x \to \zeta} \frac{u(x)}{v(x)} \quad \mbox{exists}.
$$
In view of \eqref{eq:lim1}, it follows that
$\lim_{x \to \gz} u(x)$ exists.
\end{proof}

\begin{proof}[Proof of Theorem~\ref{thm:main}]
The first part of the theorem follows from Lemma~\ref{lem:lim}

Consider now the classical cases with $p \neq d$. It follows from Theorem \ref{lindqvist} and from the existence of a positive solution $v$ satisfying $v(x) \sim 1$ near $\gz$ (see Lemma~\ref{lem:0sol}) that $\lim_{x \to \gz}u(x)\neq 0$.

Let us turn to the nonclassical cases. Suppose that $\zeta=0$ and $p>d$ (respectively, $\zeta=\infty$ and $p< d$). By Lemma~\ref{lem:lim}  $\lim_{x \to \gz} u(x)$ exists. We need to prove that the limit in \eqref{eq:ex3} is finite.
Assume to the contrary that
$$\lim_{x \to \gz}u(x)=\infty .$$

So, $\lim_{r \to \gz}m(r)=\infty$, and we are
in the situation to use the three-spheres inequality (\ref{eq:tsim}). Hence, for appropriate fixed $r_3<r_2$ (respectively, $r_2 < r_3$) and any small enough $r_1$ satisfying $0<r_1 < r_3$ (respectively, large enough $r_1$ satisfying $r_3<r_1$) we have
\begin{eqnarray}\label{eq:tsir}
m(r_1) \frac{\tilde{W}(r_2) - \tilde{W}(r_3)}{\tilde{W}(r_2) - \tilde{W}(r_1)}
+m(r_2) \frac{\tilde{W}(r_3)-\tilde{W}(r_1)}{\tilde{W}(r_2) - \tilde{W}(r_1)} \leq m(r_3)
\end{eqnarray}
for all $r_3 \in (r_1, r_2)$ (respectively, $r_3 \in (r_2,r_1)$).

After some simple algebraic manipulations we arrive at
\begin{eqnarray*}
m(r_1) \leq \frac{\tilde{W}(r_2)}{\tilde{W}(r_2) - \tilde{W}(r_3)}\, m(r_3)  & \qquad& \forall\, 0<r_1<R_1 , \left(\right.\mbox{respectively, } \forall\, r_1>R_2 \left.\right).
\end{eqnarray*}
Thus, $m(r_1)$ is bounded near $\gz$, a contradiction. Therefore,
$\lim_{x\to\zeta} u(x)$ exists and is finite.

\vskip 4mm

Assume now that $V$ satisfies also Condition~(\textbf{C2}). By Lemma~\ref{lem:0sol}, there exist two positive solutions $u_{\mathrm{small}}^{(\gz)}$ and $u_{\mathrm{large}}^{(\gz)}$ of the equation $Q'(u)=0$ defined in a punctured neighborhood of $\gz$ satisfying
\begin{equation*}
 u_{\mathrm{small}}^{(\gz)}(x) \,\underset{x\to \gz}{\sim} \,\min\{
 1 ,   v_{\alpha^*}(x)\},
\end{equation*}
and
\begin{equation*}
 u_{\mathrm{large}}^{(\gz)}(x) \,\underset{x\to \gz}{\sim} \,\max\{
 1 ,   v_{\alpha^*}(x) \}.
\end{equation*}
 These two solutions clearly satisfy the requirements of the last claim of the theorem.
\end{proof}
\begin{corollary}\label{cor1}
Let $p > d$ and $\gz=0$ (respectively, $p<d$ and $\gz=\infty$). Assume that $V$ satisfies Condition~(\textbf{C1}) near $\zeta$. Then any solution $u$ of the equation $Q'(u)=0$ in $B_R\sm \{0\}$ (respectively, $B_R^*$) which is unbounded near $\gz$  changes its sign in any punctured neighborhood of  $\gz$.

Moreover, if in addition, $V=0$, then $$\lim_{r \to \gz} M_u(r)=-\lim_{r \to \gz} m_u(r)=\infty.$$
\end{corollary}
Similarly, we have
\begin{corollary}\label{cor7}
Let $p < d$ and $\gz=0$ (respectively, $p>d$ and $\gz=\infty$). Assume that $V$ satisfies Condition~(\textbf{C1}) near $\zeta$. Then any solution $u$ of the equation $Q'(u)=0$ in $B_R\sm \{0\}$ (respectively, $B_R^*$)
  satisfying
  $$\liminf_{z\to\gz} |u(x)| =0$$
  changes its sign in any punctured neighborhood of  $\gz$.
\end{corollary}

\section{The asymptotic of positive solutions}\label{sec:asymp}
In this section we discuss the asymptotic behavior of positive solutions near a singular point $\gz$.
First, we find the asymptotic behavior near an interior singularity  for the classical case $p\leq d$ and $\gz=0$.
\begin{proof}[Proof of Theorem~\ref{thm:asymp}]
 It follows from Remarks~\ref{assumptions} and  Lemma~\ref{lem:0sol} that there exists a solution $w$, such that $w \sim v_{\alpha^*}$ near the origin.
On the other hand,  by Serrin's removable singularity result (Theorem~\ref{thm:asympSerrin}), the solution $u$ has either
a removable singularity or $u \asymp v_{\alpha^*}$. In the latter case, the ratio limit theorem
(Proposition~\ref{thm:reg}) implies that $u \sim w \sim v_{\alpha^*}$ near the origin.
\end{proof}

 We turn now to the proof of the asymptotic behavior  of positive solutions in the nonclassical cases under the integrability assumption.
\begin{proof}[Proof of Theorem~\ref{thm:asymp3}]
Let $u$ be a positive solution
of $Q'(u) = 0$ near $\zeta$. By Theorem~\ref{thm:main} we have
$\lim_{x \to \zeta} u(x) < \infty$. So we may assume that  $\lim_{x \to \zeta} u(x) = 0$ and
consequently, Theorem~\ref{thm:tsi} applies.

In particular, by fixing $r_2$ and letting  $r_1 \to 0$ (respectively, $r_1 \to \infty$) in \eqref{eq:tsiM} we obtain
\begin{equation}\label{eq:wr3leqmr3}
 M(r_3)\leq \frac{M(r_2)}{\tilde{W}(r_2)} \,\tilde{W}(r_3)
\quad 0 < r_3 < r_2\, \quad (\mbox{respectively, } r_2 <r_3 < \infty).
\end{equation}
Recall that for $V$ integrable near $\gz$ we have  $\tilde{W}\sim v_{\alpha^*}$.    Consequently, \eqref{eq:wr3leqmr3} implies $  u \leq C v_{\alpha^*}   $ near $\zeta$.

On the other hand, by Lemma~\ref{lem:0sol} there exists a positive solution $w$ near $\gz$ such that $w\sim v_{\alpha^*}$. Therefore, Lemma~\ref{lindqvist} implies that $u \sim w \sim v_{\alpha^*}$  near $\gz$.
\end{proof}

\section{Applications and conjectures}
\label{sec:app}
In this section we present some applications of Theorem~\ref{thm:main}.
First we recall the notion of positive solutions of minimal growth \cite{Agmon,PT}.
  \begin{definition}\label{def:minimal_gr} {\em
1. Let $K\Subset \Omega$, and let $u$ be a positive solution of the equation $Q'(w)=0$ in  $\Omega\sm K$. We say that $u$ is {\em a positive solution of  minimal growth in a neighborhood of infinity in $\Omega$}  if for any $K\Subset K'\Subset \Omega$ with smooth boundary  and any positive supersolution $v\in C((\Omega\sm K')\cup \,\pd K')$ of the equation $Q'(w)=0$ in $\Omega \sm K'$ satisfying  $u \leq v$ on $\pd K'$, we have $u \leq v$ in $\Omega \sm K'$.

\vskip 3mm

2. Let $u$ be a positive solution of the equation $Q'(w) = 0$ in a punctured neighborhood of $\gz$. We say that $u$ is {\em a positive solution
of minimal growth at $\zeta$} if for any smaller punctured neighborhood $K$ of
$\zeta$ such that $\pd K\sm \{\zeta\}$ is smooth, and any positive supersolution $v \in C(\bar{K} \sm \{\zeta\})$ of the equation $Q'(w) = 0$ in $K \sm \{\zeta\}$ satisfying
$u \leq v$ on $\pd K\sm \{\zeta\}$, we have $u \leq v$ in $K \sm \{\zeta\}$.
}
  \end{definition}

Assume that $V\in  L^\infty_{\mathrm{loc}}(\Omega)$, and that the equation $Q'(u)=0$ admits a  positive solution in $\Omega$. Then for any $\gz\in \Omega$ the equation $Q'(u)=0$ admits a positive solution $u_{\gz}$ of the equation $Q'(u)=0$ in $\Omega\sm \{\gz\}$ of minimal growth in a neighborhood of infinity in $\Omega$ \cite{PT}. This solution is known to be unique (up to a multiplicative constant);
see \cite{PT} for the case $1<p \leq d$, and \cite{PF} for $p >d$. The functional $Q$ is said to be {\em critical} in $\Omega$ if $u_\gz$ is in fact a positive solution of the equation $Q'(u)=0$ in $\Omega$, and {\em subcritical} otherwise.

We note that for $\Omega = \mathbb{R}^d$, the solutions of minimal growth in the neighborhood of infinity of $\Omega$ are solutions with minimal growth at $\zeta = \infty$.

For solutions of minimal growth at $\gz$ we have the following asymptotic behavior.
\begin{theorem}\label{thm_nee}
Assume that $V$ satisfies conditions~(\textbf{C1}) and (\textbf{C2}) near $\gz$ and $p \neq d$. Let $u$ be solution
of the equation $Q'(w) = 0$ of minimal growth at $\zeta$. Then
$$
u(x) \;\underset{x\to \gz}{\sim}\; \min\{1,|x|^{\alpha^*}\}.
$$
\end{theorem}
\begin{proof}
In the classical cases the theorem is just a reformulation of Theorem~\ref{thm:main}.
In the nonclassical cases the existence of a solution with the required asymptotic behavior near $\gz$ is guaranteed
by Lemma~\ref{lem:0sol}, and therefore the asymptotic behavior of $u$  near $\gz$ follows from Lemma~\ref{lindqvist}.
\end{proof}

In a recent paper \cite{JV}, Jaye and Verbitsky study the  behavior of the unique positive solution $u$ of  the equation $Q'(u)=0$ in $\mathbb{R}^d$ of minimal growth in a neighborhood of infinity in $\mathbb{R}^d$ with an isolated singularity at the origin, where $p<d$. It is proved that   if $V$ is a {\em nonpositive} potential which belongs to a certain class of measures, then $u(x)\asymp  |x|^{\ga^*}$
in $\mathbb{R}^d$. Our results here allows for following generalization.

\begin{corollary}\label{JV}
Let $\Omega = \mathbb{R}^d$ and $p < d$. Assume that $V$ belongs to $L^{\ga}(B_R)$ with $\ga>d/p$, and $V$ satisfies
conditions (\textbf{C1}) and (\textbf{C2}) at infinity. Suppose that $Q$ is subcritical in  $\mathbb{R}^d$.
Then the unique positive solution $u$ of the equation $Q'(u) = 0$ in
$\Omega \sm \{0\}$ of minimal growth in a neighborhood of infinity in $\Omega$
satisfies $u \asymp |x|^{\alpha^*}$.
\end{corollary}
\begin{proof}
According to Theorem~\ref{thm:asympSerrin} $u \asymp |x|^{\alpha^*}$ near zero, and
according to Theorem~\ref{thm_nee} $u \sim |x|^{\alpha^*}$ near infinity.
Hence the claim follows by a compactness argument.
\end{proof}

The next result answers a question posed
in \cite[Section 5]{PT}, by showing that for a general domain $\Omega$, and  $p>d$, positive solutions
of the equation $Q'(u)=0$ in $\Omega\sm \{\gz\}$ of minimal growth in a neighborhood of infinity in $\Omega$
are comparable to $1$ near the isolated singular point $\gz$.
\begin{theorem} \label{thm:min}
Let $\Omega$ be a domain in $\mathbb{R}^d$. Fix $\gz \in \Omega$; without loss of generality assume $\zeta=0$.  Suppose that $V$ satisfies Condition~(\textbf{C1}) with respect to $\zeta=0$. Assume further that the equation  $Q'(u) = 0$ admits a positive
solution in $\Omega$, and let $u_0$ be a positive solution of $Q'(u) = 0$ in $\Omega\sm\{0\}$
of minimal growth in a neighborhood of infinity in $\Omega$.
If $p > d $, then there exists a positive constant $C$ such that
\begin{equation}\label{C}
\lim_{x \to 0} u_0 (x) = C>0.
\end{equation}
 If $p< d$, then \eqref{C} holds if and only if $Q$ is critical in $\Gw$.
\end{theorem}
\begin{proof}
Suppose that $p > d $. By Theorem~\ref{thm:main}, $u_0$ is in fact continuous at $0$.
It remains to prove that $u_0(0) >0$. Assume to the contrary that $u_0(0) = 0$,
and let $v$ be a positive solution of $Q'(u) = 0$ in $\Omega$ satisfying $v(0) = 1$. Then
for any $\varepsilon>0$ there exists $r=r(\varepsilon)>0$  such that
$$
 u_0(x)\leq \varepsilon v(x)  \qquad  \forall \quad x \in S_r,
$$
and by Definition~\ref{def:minimal_gr},
$$
u_0(x)\leq \varepsilon v(x)  \qquad  \forall x \in \Omega \sm B_r.
$$
Clearly, $r(\varepsilon) \to 0$ as $\varepsilon \to 0$. Consequently, $u_0= 0$, a contradiction (cf. \cite[Section 5]{PT}).

The last statement of the theorem follows from Corollary~\ref{thm:remov}.
\end{proof}
The following example illustrates Theorem~\ref{thm:min}.
\begin{example}\label{ex:mg}{\em
Let $\Omega=B_R$, $Q'(u)=-\pl (u)$, $p>d$, and $\gz=0$. The corresponding positive solution $u_0$ of  minimal growth in a neighborhood of infinity in $\Omega$ is given by
$$u_0(x):=R^{\alpha^*}-|x|^{\alpha^*} \qquad x\in B_R\sm \{0\}.$$
Note that $\lim_{x\to 0}u_0(x)=R^{\alpha^*}>0$.
}
\end{example}

Next, we present a positive Liouville theorem in $\mathbb{R}^d$, where $p<d$. This result is well known for $p=2$ under weaker assumptions (see for example, \cite{P}).
\begin{theorem}\label{thm:Liouville}
 Let $p< d$, and  suppose that $V\in L^\infty_{\mathrm{loc}}(\mathbb{R}^d)$ satisfies conditions~(\textbf{C1}) and (\textbf{C2}) near infinity. Assume further that
\begin{equation}\label{Q1} Q(u):=\int_{\mathbb{R}^d}\left(|\nabla u|^p+V|u|^p\right)\,\mathrm{d}x\geq 0
\qquad \forall  u\in C_0^\infty(\mathbb{R}^d).
\end{equation}
Then the equation $Q'(u)=0$ admits a unique (up to a multiplicative constant) positive solution  $u$ in $\mathbb{R}^d$. Moreover, there exists $C\geq 0$ such that
$$\lim_{x\to\infty}u(x)=C,$$
and $C=0$ if $Q$ is critical in $\mathbb{R}^d$. Furthermore, if $\,V$ is integrable near infinity, then  $C=0$ if and only if $Q$ is critical in $\mathbb{R}^d$.
\end{theorem}
\begin{proof}
Assumption \eqref{Q1} implies that the equation $Q'(u)=0$ admits an entire positive solution $u$ in $\mathbb{R}^d$ \cite{PTalt}. The {\em uniqueness} follows from \cite[Theorem~2.6]{PF}, where a positive Liouville theorem is proved under weaker assumptions. The existence of $ \lim_{x\to\infty}u(x)=C <\infty$ follows from Theorem~\ref{thm:main}.

By Theorem~\ref{thm_nee}, a positive solution $u$ of the equation $Q'(u)=0$ of minimal growth in a neighborhood of infinity in $\mathbb{R}^d$ satisfies $u\sim v_{\alpha^*}$, and in particular $u$ tends to zero as $x\to\infty$. Thus, if $Q$ is critical in $\mathbb{R}^d$, then $C=0$.

Moreover, if $V$ is integrable, then by Theorem~\ref{thm:asymp3} we have $\lim_{x\to\infty}u(x)=0$ if and only if
\begin{equation}\label{eq:conj1}
    u(x) \;\underset{x\to \infty}{\sim}\; |x|^{\alpha^*},
\end{equation}
which takes place if and only if $u$ is a positive solution of  minimal growth in a neighborhood of infinity in $\mathbb{R}^d$.
\end{proof}
\begin{remark}\label{rem:13}{\em
 Under the assumptions of Theorem~\ref{thm:Liouville}, the uniqueness of a positive entire solution (up to a multiplicative constant) can also be obtained using the existence of the limit $ \lim_{x\to\infty}u(x)=C$.
  }
\end{remark}

 We conclude our paper with a conjecture. Let $\zeta\in\{0,\infty\}$ be an isolated singular point of the equation $Q'(w)=0$ in a domain $\Gw$. Denote by $\mathcal{G}_\gz$ the germ of all positive solutions $u$ of the equation $Q'(w)\!=\!0$ in some punctured neighborhood $\Gw'\subset \Gw$ of $\zeta$ (The neighborhood $\Gw'$ might depend on $u$).
Let $u,v\in \mathcal{G}_\gz$. We use the following notations.
\begin{itemize}
\item
We denote $u\underset{x\to \zeta}{\sim} v\;$ if
$\;\displaystyle{\underset{x \in \Gw}{\underset{x \to \zeta}{\lim}}\,\frac{u(x)}{v(x)}= C}$ for some  positive constant $C$.
\item By $u \underset{x \to \zeta}{\prec}v$ we mean that $\displaystyle{\underset{x \in \Gw}{\underset{x \to \zeta}{\lim}}\,\frac{u(x)}{v(x)} =0}$.

 \item By $u \underset{x \to \zeta}{\lsim} v $ we mean that either $u \underset{x \to \zeta}{\sim} v$ or $u \underset{x \to \zeta}{\prec}v$.

\item  We denote $u\underset{x \to \zeta}{\succ} u\;$ if $\;v \underset{x \to \zeta}{\prec}u$. Similarly,  $u\underset{x \to \zeta}{\gsim} v\;$ if $\;v \underset{x \to \zeta}{\lsim} u$.
     \end{itemize}
Clearly, $u \,\underset{x \to \zeta}{\sim}\, v$ defines an equivalence relation and equivalence classes on $\mathcal{G}_\gz$.
\begin{definition}\label{def:regular} {\em
We say that $\zeta$ is a {\em regular point of the equation $Q'(w)=0$ in $\Gw$}   if for any two positive solutions $u,v\in \mathcal{G}_\gz$ we have  either  $u \underset{x \to \zeta}{\lsim} v$ or $u \underset{x \to \zeta}{\gsim} v$.
 }
\end{definition}
\begin{conjecture}\label{main_conj}
Suppose that \eqref{eq:1} admits a (global) positive solution and $V$ has a Fuchsian type singularity at the isolated singular point $\zeta$. Then
\begin{enumerate}
\item[i)] $\zeta$ is a regular point of equation (\ref{eq:1}).
\item[ii)]  Equation \eqref{eq:1} admits a unique (global) positive solution of minimal growth in a neighborhood of infinity in $\Omega \sm \{\zeta\}$.
\item[iii)] $\mathcal{G}_\gz$ admits exactly two equivalence classes with respect to $\sim$\,.
\end{enumerate}
   \end{conjecture}
 We note that the results discussed in the present paper and in \cite{PF,P94} give partial answers to Conjecture~\ref{main_conj}.  In particular, in \cite{PF} the authors proved that {\em i)} implies {\em ii)}, and proved the regularity for
potentials with a weak Fuchsian isolated singularity and for spherically symmetric potentials with a Fuchsian isolated singularity.


\begin{center}{\bf Acknowledgments} \end{center}
The authors wish to thank V.~Liskevich for valuable discussions, and the referee for his helpful remarks.
 They acknowledge the support of the Israel Science
Foundation (grant No. 587/07) founded by the Israel Academy of
Sciences and Humanities. M.~F. acknowledges also the support of the Israel Science
Foundation (grant No. 419/07) founded by the Israel Academy of
Sciences and Humanities. M.~F. was also partially supported by a fellowship
of the UNESCO fund. This research was also supported by Sidney \& Ann Grazi Research Fund.



\begin{thebibliography}{1}
\newcommand{\cs}{\vspace{2\parskip}}

\bibitem{Agmon1} S.~Agmon,
{\em Unicit\'{e} et convexit\'{e} dans les probl\`{e}mes diff\'{e}rentiels},  S\'{e}minaire de Math\'{e}\-matiques Sup\'{e}rieures No. 13,   Les Presses de l'Universit\'{e} de Montr\'{e}al, Montreal, Que., 1966.

\bibitem{Agmon}
S.~Agmon, On positivity and decay of solutions of second order elliptic
  equations on {R}iemannian manifolds,
in: {\em Methods of Functional Analysis and Theory of Elliptic
  Equations ({N}aples, 1982)}, Liguori, Naples, 1983, pp. 19--52.

\bibitem{bhat}
T.~Bhattacharya, On the behaviour of $\infty$-harmonic functions near isolated
points, {\em Nonlinear Anal.} 58 (2004) 333--349.

\bibitem{veron}
M.~F.~Bidaut-V{\'e}ron,  R.~Borghol, L.~V{\'e}ron,
Boundary Harnack inequality and a priori estimates of singular
  solutions of quasilinear elliptic equations,
{\em Calc. Var. Partial Differential Equations} 27 (2006) 159--177.


\bibitem{Brum} R.~Brummelhuis, Three-spheres theorem for second order elliptic equations,  {\em J. Anal. Math.} 65 (1995) 179--206.

\bibitem{PF}
M.~Fraas, Y.~Pinchover, Positive Liouville theorems and asymptotic behavior for $p$-Laplacian type elliptic equations with a Fuchsian potential, {\em Confluentes Mathematici} 3 (2011) 291--323.

\bibitem{Garcia}
J.~Garc{\'{\i}}a-Meli{\'a}n,  J.~Sabina~de Lis,
Maximum and comparison principles for operators involving the
  {$p$}-{L}aplacian,
 {\em J. Math. Anal. Appl.} 218 (1998) 49--65.

\bibitem{GZ} F.~Gesztesy, Z.~Zhao, On positive solutions of critical Schr\"odinger operators in two dimensions,  {\em J. Funct. Anal.}  127 (1995) 235--256.


\bibitem{Martio} J.~Heinonen, T.~Kilpel\"{a}inen,  O.~Martio,
  {\em Nonlinear Potential Theory of Degenerate Elliptic Equations}, unabridged republication of the 1993 original,  Dover, Mineola, N.~Y., 2006.

\bibitem{JV} B.~J.~Jaye, I.~E.~Verbitsky, The fundamental solution of nonlinear equations with natural growth terms, to appear in {\em Ann. Sc. Norm. Super. Pisa Cl. Sci.}, arXiv1002.4664

\bibitem{KV}  S.~Kichenassamy, L.~V\'{e}ron, Singular solutions of the te $p\,$-Laplace equation,
{\em  Math. Ann.} 275 (1986) 599--615, and 277 (1987) 352.

\bibitem{Landis} E.~M.~Landis, Some problems of the qualitative theory of second-order elliptic equations (case of several independent variables), {\em Uspehi Mat. Nauk}  18 (1963) 3--62.

\bibitem{Lindqvist} P.~Lindqvist, On the equation $\mathrm{div}(|\grad u|^{p-2} \grad u) + \lambda |u|^{p-2} u =0$,
{\em Proc. Amer. Math. Soc.} 109 (1990) 157--164.

\bibitem{LS} V.~Liskevich, I.~I.~Skrypnik, Isolated singularities of solutions to quasi-linear elliptic equations with absorption, {\em J. Math. Anal. Appl.}  338 (2008) 536--544.

\bibitem{LSS} V.~Liskevich, I.~I.~Skrypnik, I.~V.~Skrypnik, Positive solutions to singular
nonlinear Schr\"odinger-type equations, {\em Ukr. Mat. Visn.}  6 (2009) 55--96.

\bibitem{MH} F.~I.~Mamedov, A.~Harman,
On the removability of isolated singular points for degenerating nonlinear elliptic equations,
{\em Nonlinear Anal.} 71 (2009) 6290--6298.

\bibitem{Mikl} V.~M.~Miklyukov, A.~Rasila, M.~Vuorinen, Three spheres theorem for $p$-harmonic functions,  {\em Houston J. Math.}  33 (2007) 1215--1230.

\bibitem{P}
Y.~Pinchover, On the equivalence of Green functions of second order elliptic equations in $\mathbb{R}^n$,  {\em Differential Integral Equations}  5 (1992) 481--493.

\bibitem{P94}
Y.~Pinchover, On positive Liouville theorems and
asymptotic behavior of solutions of Fuchsian type elliptic
operators, {\em Ann.  Inst.  H. Poincar\'{e}. Anal. Non Lin\'{e}aire}
11 (1994) 313--341.

\bibitem{PTalt}
Y.~Pinchover, K.~Tintarev,
Ground state alternative for {$p$}-{L}aplacian with potential term,
{\em Calc. Var. Partial Differential Equations} 28 (2007) 179--201.

\bibitem{PT}
Y.~Pinchover, K.~Tintarev,
On positive solutions of minimal growth for singular
  {$p$}-{L}aplacian with potential term,
{\em Adv. Nonlinear Stud.} 8 (2008) 213--234.

\bibitem{PW} M.~H.~Protter, H.~F.~Weinberger, {\em Maximum Principles in Differential Equations},
Springer-Verlag, New York, 1984.

\bibitem{Puc}
P.~Pucci, J.~Serrin,
{\em The Maximum Principle},
 Progress in Nonlinear Differential Equations and their Applications
  73, Birkh\"auser-Verlag, Basel, 2007.

\bibitem{Qi} W.~J.~Qi, Positive entire solutions for singular $p$-Laplacian equations on $\mathbb{R}^N$ with $p\geq N\geq 2$. {\em Nonlinear Anal.} 73 (2010) 1645--1652.

\bibitem{Ser}
J.~Serrin,
Isolated singularities of solutions of quasi-linear equations,
{\em Acta Math.} 113 (1965) 219--240.

\bibitem{Serrin2}
J.~Serrin, Singularities of solutions of nonlinear equations, {\em Proc. Sympos. Appl. Math.} Vol. XVII, Amer. Math. Soc., Providence, R.I., (1965), pp. 68--88.

\bibitem{V} R.~V\'{y}born\'{y}, The Hadamard three-circles theorems for partial differential equations,  {\em Bull. Amer. Math. Soc.}  80 (1973) 81--84.

\end{thebibliography}
\end{document}